\documentclass[12pt]{amsart}

\usepackage[dvips]{graphicx}
\usepackage{amssymb,amsmath,amsthm}

\def\Pd{{{\mathcal M}_d(P)}}
\def\Ptwo{{\mathcal M}_2(P)}

\def\RR{{\mathbb R}}
\def\ZZ{{\mathbb Z}}

\newtheorem{Theorem}{Theorem}[section]
\newtheorem{Lemma}[Theorem]{Lemma}

\newtheorem{Proposition}[Theorem]{Proposition}

\theoremstyle{definition}

\newtheorem{Example}[Theorem]{Example}

\begin{document}

\title{A Gr\"obner basis characterization for chordal comparability graphs}
\author{Hidefumi Ohsugi and Takayuki Hibi}

\address{Hidefumi Ohsugi,
Department of Mathematical Sciences,
School of Science and Technology,
Kwansei Gakuin University,
Sanda, Hyogo 669-1337, Japan} 
\email{ohsugi@kwansei.ac.jp}

\address{Takayuki Hibi,
Department of Pure and Applied Mathematics,
Graduate School of Information Science and Technology,
Osaka University,
Suita, Osaka 565-0871, Japan}
\email{hibi@math.sci.osaka-u.ac.jp}


\begin{abstract}
In this paper, we study toric ideals associated with multichains of posets.
It is shown that the comparability graph of a poset is chordal 
if and only if there exists a quadratic Gr\"obner basis of the toric ideal
of the poset.
Strong perfect elimination orderings of strongly chordal graphs play
an important role.
\end{abstract}

\maketitle

\section*{Introduction}

An $n \times m$ integer matrix $A = ({\bf a}_1, \ldots, {\bf a}_m)$ 
is called a {\it configuration} if there exists ${\bf c} \in \RR^n$ such that
${\bf a}_j \cdot {\bf c} = 1$ for $ 1 \le j \le m$.
Let $K[y_1,\ldots , y_m]$ be a polynomial ring in $m$ variables over a field $K$.
Given a configuration $A$, the binomial ideal
$$
I_A
=
\left<
\prod_{b_i >0}
y_i^{b_i}
-
\prod_{b_j <0}
y_j^{-b_j} 
\in
K[y_1, \ldots, y_m]: 
 {\bf b} = 
\begin{pmatrix}
b_1\\
\vdots\\
b_m
\end{pmatrix}
\in \ZZ^m , A {\bf b} = {\bf 0} 
\right>
$$
is called the {\em toric ideal} of $A$.
Any toric ideal is generated by homogeneous binomials,
and has a Gr\"obner basis consisting of homogeneous binomials.
See \cite{dojoEN, Sturmfels} for basics on toric ideals.
Each of the following is 
one of the most important and fundamental problems on toric ideals:
\begin{itemize}
\item[(a)]
Is the toric ideal $I_A$ generated by quadratic binomials?
\item[(b)]
Does there exist a monomial order such that
a Gr\"obner basis of $I_A$ consists of quadratic binomials?
\end{itemize}
Note that any Gr\"obner basis of $I_A$ is a set of generators of $I_A$.
These problems arise in the study of Koszul algebras.
The algebra $K[y_1,\ldots, y_m] / I_A$ is said to be {\em Koszul}
if the minimal graded free resolution of $K$ as a $K[y_1,\ldots, y_m] / I_A$-module is linear.
It is known that
\begin{eqnarray*}
I_A 
\mbox{ has a quadratic Gr\"obner basis }
&\Longrightarrow&
K[y_1,\ldots, y_m] / I_A
\mbox{ is Koszul }\\
&\Longrightarrow&
I_A \mbox{ is generated by quadratic binomials}
\end{eqnarray*}
holds in general.
However, all of the converse implications are false.
See, e.g., \cite{OH2}.
Problems (a) and (b) are studied for configurations
arising from various kinds of combinatorial objects.
The following is a partial list of them:
\begin{enumerate}
\item
Toric ideals arising from order polytopes of finite posets
\cite{HibiGor};

\item
Toric ideals arising from cut polytopes of finite graphs
\cite{Eng, NP};

\item
Toric ideals of the vertex-edge incidence matrix of finite graphs
\cite{OH1, OH2};

\item
Toric ideals arising from graphical models
\cite{D, GMS};

\item
Toric ideals arising from matroids
\cite{Bla, Blu, Kas, LM}.
\end{enumerate}
In particular, one of the most famous open problems on toric ideals 
is White's conjecture \cite{White}:
He conjectured that the toric ideal arising from 
any matroid is generated by some quadratic binomials.

In the present paper, we study toric ideals associated with multichains of posets.
Let $d \ge 2$ be an integer and let $P=\{x_1,\ldots, x_n\}$ be a poset.
We associate a multichain
$
C: x_{i_1} \le x_{i_2} \le \cdots \le x_{i_d}
$
of length $d-1$ with a (not necessarily $(0,1)$) vector
$
\rho(C) = {\bf e}_{i_1} + {\bf e}_{i_2} + \cdots + {\bf e}_{i_d} \in \ZZ^n
$,
where ${\bf e}_i$ is the $i$th unit vector in $\RR^n$.
We often regards $C$ as a multiset $\{x_{i_1}, x_{i_2}, \ldots, x_{i_d}\}$.
Let $\Pd = \{C_1,\ldots, C_m\}$ be a set of multichains of $P$ of length $d-1$.
Then the toric ideal $I_{\Pd}$ of $\Pd$ is 
the toric ideal of the configuration $(\rho(C_1),\ldots, \rho(C_m))$.
For example, if $d=3$ and $P = \{x_1, x_2, x_3\}$ is a poset whose maximal chains are $x_1 > x_2$ and $x_2 < x_3$,
then the corresponding configuration is
$$
\begin{pmatrix}
3 & 2  & 1   & 0 & 0 & 0 & 0\\
0 & 1  & 2   & 3 & 2 & 1 & 0\\
0 & 0  & 0   & 0 & 1 & 2 & 3
\end{pmatrix}.
$$
For any $d \ge 2$ and chain 
$P=\{x_1,\ldots, x_n\}$ of length $n-1$,
it is known that
$I_\Pd$ is the toric ideal of the 
{\it $d$th Veronese subring} of a polynomial ring in $n$ variables, and
$I_{\Pd}$ has a quadratic Gr\"obner basis.
Thus, in general, $I_{\Pd}$ is a toric ideal of a subconfiguration
of the $d$th Veronese subring.
There are several results on toric ideals of subconfigurations
of the $d$th Veronese subring:
algebras of Veronese type \cite[Theorem 14.2]{Sturmfels}
and
algebras of Segre--Veronese type \cite{OH2, AHOT}.
However, the results of the present paper are different from 
these results.
The toric ideal of algebras of Veronese / Segre--Veronese type 
has a squarefree initial ideal.
On the other hand, $I_{\Pd}$ has no squarefree initial ideal
except for some trivial cases (Proposition \ref{sqfnormal}).

This paper is organized as follows. 
In Section 1, it is shown that the comparability graph of a poset is chordal 
if and only if there exists a quadratic Gr\"obner basis of the toric ideal
of the poset (Theorem \ref{main}).
In order to construct a quadratic Gr\"obner basis,
the most difficult point is to find a suitable monomial order on a polynomial ring.
Strong perfect elimination orderings of strongly chordal graphs play
an important role in overcoming this difficulty.
In Section 2, 
we apply the results in Section 1 to a toric ring arising from a graph.
Given a graph $G$, let $A_G$ be the vertex-edge incidence matrix of $G$ and 
let $E_n$ be an identity matrix.
It is proved that the toric ideal of the configuration $(\ 2 E_n \ | \ A_G \ )$ has  a quadratic Gr\"obner basis
if and only if $G$ is strongly chordal (Theorem \ref{stronglychordal}).

\section{A Gr\"obner basis characterization}

In this section, we give the main theorem of this paper and its proof.
First we present a useful lemma.

\begin{Lemma}
\label{lemma1}
Let $A = ({\bf a}_1, \ldots, {\bf a}_m)$ be a configuration.
Suppose that
${\bf a}_{i_1} +\cdots + {\bf a}_{i_r} = {\bf a}_{j_1} +\cdots +{\bf a}_{j_r} $ $(r \ge 3)$,
where $\{i_1,\ldots, i_r\} \cap \{j_1,\ldots, j_r\} = \emptyset$.
If 
$$
{\bf a}_{i_k} +{\bf a}_{i_\ell} = {\bf a}_p + {\bf a}_q
\Longleftrightarrow
\{i_k, i_\ell\} = \{p,q\}
$$
holds for any $1 \le k < \ell \le r$ and $1 \le p, q \le m$,
then $I_A$ is not generated by quadratic binomials.
\end{Lemma}

\begin{proof}
Suppose that
${\bf a}_{i_1} +\cdots + {\bf a}_{i_r} = {\bf a}_{j_1} +\cdots +{\bf a}_{j_r} $ $(r \ge 3)$,
where $\{i_1,\ldots, i_r\} \cap \{j_1,\ldots, j_r\} = \emptyset$.
Then, $A ({\bf e}_{i_1} +\cdots + {\bf e}_{i_r} - {\bf e}_{j_1} -\cdots -{\bf e}_{j_r} )={\bf 0}$, and hence, $y_{i_1} \cdots y_{i_r} - y_{j_1} \cdots y_{j_r} $ belongs to $I_A$.
Let $f = y_{i_1} \cdots y_{i_r} - y_{j_1} \cdots y_{j_r} $.
Since $\{i_1,\ldots, i_r\} \cap \{j_1,\ldots, j_r\} = \emptyset$, 
$f$ is a nonzero binomial (of degree $r \ge 3$).
If $f$ is generated by quadratic binomials in $I_A$,
then there exists a quadratic binomial $g = y_{i_k} y_{i_\ell}  - y_p y_q$ $(\neq 0)$
belonging to $I_A$ such that
$1 \leq k< \ell \le r$ and $1 \le p,q \le m$.
Since $g$ belongs to $I_A$, we have ${\bf a}_{i_k} +{\bf a}_{i_\ell} = {\bf a}_p + {\bf a}_q$.
By assumption, $\{i_k, i_\ell\} = \{p,q\}$, and hence $g=0$.
This is a contradiction.
Thus, $f$ is not generated by quadratic binomials in $I_A$.
\end{proof}

Let $G$ be a finite simple graph on the vertex set $\{v_1,\ldots,v_n\}$
whose edge set is $E(G)$. 
Given a vertex $v$ of $G$,
let $N(v)$ denote the induced subgraph of $G$
consisting of all vertices adjacent to $v$.
A vertex $v$ of $G$ is called {\it simplicial} in $G$
if $N(v)$ is a clique in $G$.
The ordering $v_1,\ldots,v_n$ of the vertices of $G$
is called a {\it perfect elimination ordering} of $G$
if, for all $i \in \{1,2,\ldots, n-1\}$, the vertex $v_i$ is simplicial in the induced subgraph of $G$
on vertices $v_i, v_{i+1},\ldots,v_n$.
A graph $G$ is called {\it chordal} if
the length of any induced cycle of $G$ is three.
It is known that a graph $G$ is chordal if and only if 
$G$ has a perfect elimination ordering. 
Several interesting results on commutative algebra
related with chordal graphs are known (e.g., \cite{D, F, HHZ}).
A perfect elimination ordering $v_1,\ldots,v_n$ of a graph $G$
 is called a {\it strong perfect elimination ordering}
 if one of the following equivalent conditions holds: 
\begin{itemize}
\item[(i)]
If $i<j<k< \ell$ and $\{v_i , v_k\}, \{v_i, v_\ell\}, \{v_j, v_k\} \in E(G)$,
then $\{v_j, v_\ell\} \in E(G)$;
\item[(ii)]
If $i<j$ and $k< \ell$ with
$
\{v_i , v_k\}, \{v_i, v_\ell\}, \{v_j, v_k\}\! \in \! E(G)$, then
$\{v_j, v_\ell\}\! \in \! E(G)$.
\end{itemize}
A graph $G$ is called {\it strongly chordal}
if $G$ has a strong perfect elimination ordering.

Let $P=\{ x_1, \ldots , x_n \}$ be a poset.
Then, the {\em comparability graph} $G_P$ of $P$ 
is a graph on the vertex set $P$ such that
$\{x_i,x_j\}$ is an edge of $G_P$ 
if and only if $x_i < x_j$ or $x_j < x_i$.
It is known that if $G_P$ is chordal, then $G_P$ is strongly chordal.
See, e.g., \cite{graphbook} for details.
Suppose that  the comparability graph $G_P$ of a poset $P$ is chordal.
Assume that $x_1, \ldots , x_n$ is a strong perfect elimination ordering of $G_P$.
Let $\Pd = \{C_1, \ldots, C_m\}$, where $\rho(C_i)- \rho(C_j)
 = (0,\ldots, 0, \alpha^{(i,j)}, \ldots)$ with $\alpha^{(i,j)} >0$
for all $1 \le i<j \le m$.
Recall that $I_\Pd \subset K[y_1, \ldots, y_m]$ is the toric ideal of a configuration $(\rho(C_1), \ldots, \rho(C_m))$,
where each $y_i$ corresponds to $C_i$.
Let $<_{\rm rev}$ denote the reverse lexicographic order
induced by the ordering $y_1 < \cdots < y_m$.

Now we are in the position to state the main theorem of the present paper.

\begin{Theorem}
\label{main}
Let $P$ be a poset.
Then the following conditions are equivalent:
\begin{itemize}
\item[{\rm (i)}]
The comparability graph $G_P$ of $P$ is chordal.
\item[{\rm (ii)}]
The toric ideal $I_{\Pd}$ is generated by quadratic binomials for some $d$;
\item[{\rm (iii)}]
The toric ideal $I_{\Pd}$ has a quadratic Gr\"obner basis for some $d$;
\item[{\rm (iv)}]
The toric ideal $I_{\Pd}$ is generated by quadratic binomials for all $d \ge 2$;
\item[{\rm (v)}]
The toric ideal $I_{\Pd}$ has a quadratic Gr\"obner basis for all $d \ge 2$.
\end{itemize}
\end{Theorem}

\begin{proof}
The implications
(v) $\Longrightarrow$ (iv)$\Longrightarrow$ (ii)
and
(v) $\Longrightarrow$ (iii)$\Longrightarrow$ (ii)
are trivial.
We will show (ii) $\Longrightarrow$ (i) and (i) $\Longrightarrow$ (v).

{\bf (ii) $\Longrightarrow$ (i)}
Suppose that $G_P$ has an induced cycle $C$ of length $\ge 4$.
Since $G_P$ is a comparability graph, $C$ is an even cycle.
Let $C=(x_1,\ldots, x_{2 \ell})$ with $\ell \ge 2$
and let $C_1 = \{x_1, \ldots, x_1\}$, $C_2 =\{x_1,\ldots,x_1, x_2\}$,
$C_3 =\{x_2,x_3, \ldots, x_3\}$, 
$C_4 =\{x_3, \ldots, x_3\}$,
$C_{2i+1}= \{x_{2i-1},x_{2i}, \ldots, x_{2i}\}$ $ (2 \le i \le \ell) $, 
$C_{2i+2}= \{x_{2i}, \ldots, x_{2i},x_{2i+1}\}$ $(2\le i \le \ell-1)$, and 
$C_{2\ell +2} =\{x_{2\ell},\ldots,x_{2 \ell}, x_1\}$
be multichains in $\Pd$.
Then, we have
\begin{eqnarray*}
\sum_{k=1}^{\ell+1} \rho( C_{2k-1} )
&=& 
d {\bf e}_1 + {\bf e}_2 + (d-1) {\bf e}_3
+
\sum_{i=2}^{\ell}
({\bf e}_{2i-1} + (d-1) {\bf e}_{2i})\\
&=& 
(d-1) {\bf e}_1 + {\bf e}_2 + d {\bf e}_3
+
\left(
\sum_{i=2}^{\ell-1}
((d-1) {\bf e}_{2i} + {\bf e}_{2i+1})
\right)
+ (d-1) {\bf e}_{2\ell} + {\bf e}_1
\\
&=& \sum_{k=1}^{\ell+1} \rho( C_{2k} ) .
\end{eqnarray*}
Since $C$ is an induced cycle of $G_P$,
it follows that,
for 
$C', C'' \in \Pd$, $\rho(C_i) + \rho(C_j) = \rho(C') + \rho(C'')$
if and only if
$ \{C_i , C_j\}=\{C', C'' \}$.
By Lemma \ref{lemma1}, $I_\Pd$ is not generated by quadratic binomials.

{\bf (i) $\Longrightarrow$ (v)}
Let ${\mathcal G}$  be the reduced Gr\"obner basis of $I_{\Pd}$ with respect to 
the reverse lexicographic order $<_{\rm rev}$ defined as above.
Suppose that there exists a binomial $g = y_{u_1} \cdots y_{u_\beta} -y_{v_1} \cdots y_{v_\beta} \in {\mathcal G}$ of degree $\beta \ge 3$
whose initial monomial is $y_{u_1} \cdots y_{u_\beta}$.
By \cite[Lemma 4.6]{Sturmfels}, it follows that
$\{u_1,\ldots, u_\beta\} \cap \{v_1,\ldots, v_\beta\} = \emptyset$.
Since $g$ belongs to $I_\Pd$, we have $\sum_{i=1}^\beta \rho (C_{u_i}) = \sum_{i=1}^\beta \rho (C_{v_i}) $.
Let $y_k$ be the smallest variable in $g$.
Then, $k \in \{v_1,\ldots, v_\beta\}$.
Let $C_k = \{x_{i_1}, \ldots, x_{i_d}\}$, where $i_1 \le \cdots \le i_d$.

Since $- \rho(C_k) + \sum_{i=1}^\beta \rho (C_{u_i}) $ is nonnegative 
and since $y_k$ is the smallest variable in $g$,
there exists a variable $y_{k_1}$ such that
$
C_{k_1} = \{x_{i_1}, \ldots x_{i_p}, x_{\ell_1},\ldots, x_{\ell_q} \}
$
$
(i_1 \le \cdots \le i_p \le  \ell_1 \le \cdots \le \ell_q,
\ \ 1 \le p <d, \mbox{ and } i_{p+1} <\ell_1)
$
and that $k_1 \in \{u_1,\ldots, u_\beta\} $.
Since $x_{i_1}$ is simplicial, 
it follows that
$P' = \{x_{i_1},\ldots, x_{i_d} , x_{\ell_1},\ldots, x_{\ell_q} \}$ is a multichain of $P$.
On the other hand, since $-\rho(C_{k_1})- {\bf e}_{i_{p+1}} +\sum_{i=1}^\beta \rho (C_{u_i}) $
is nonnegative, 
there exists a variable $y_{k_2}$ such that
$C_{k_2} =
\{x_{i_{p+1}}, x_{ j_2}, \ldots ,x_{j_d} \}
$
$
(j_2 \le \cdots \le j_d)
$
and that $y_{k_1} y_{k_2}$ divides $y_{u_1} \cdots y_{u_\beta}$.
Since $y_k <y_{k_2}$, we have
$i_1 \le i_{p+1}, j_2$, and in addition,
at least one of $i_{p+1}$ and $j_d$ is greater than $i_1$.
It is enough to show that there exists a quadratic binomial
$f$ ($\neq 0$) in $I_{\Pd}$ whose initial monomial is
$y_{k_1} y_{k_2}$,
which yields a contradiction.

\noindent
{\bf Case 1.} ($i_1 \in \{i_{p+1}, j_2\}$.)
Since $x_{i_1}$ is simplicial and since
$$i_1 = \min\{i_1, \ldots, i_p , \ell_1, \ldots, \ell_q \} = \min\{ i_{p+1}, j_2,\ldots, j_d \},$$
$\{x_{i_1},\ldots, x_{i_p} , x_{\ell_1},\ldots, x_{\ell_q} \} \cup \{ x_{i_{p+1}}, x_{j_2}, \ldots,  x_{j_d} \}$ 
is a multichain of $P$.
Let $\alpha_1 ,\ldots, \alpha_{2d}$ be integers such that
$\alpha_1 \le \cdots \le \alpha_{2d}
$
and
$$
\{x_{i_1}, \ldots, x_{i_p}, x_{\ell_1}, \ldots, x_{\ell_q} , x_{i_{p+1}}, x_{j_2}, \ldots, x_{j_d} \}
=
\{x_{\alpha_1}, \ldots, x_{\alpha_{2d}}\}
$$
as multisets.
Then, $f= y_{k_1} y_{k_2} - y_{k_3} y_{k_4}$,
where 
$C_{k_3} = \{x_{\alpha_1}, \ldots, x_{\alpha_d}\}$ and
$C_{k_4} = \{x_{\alpha_{d+1}}, \ldots, x_{\alpha_{2d}}\}$,
belongs to $I_{\Pd}$.
Suppose that $f =0$.
Then, either $i_1 = \cdots = i_p = \ell_1 = \cdots =\ell_q$ or 
$i_1 = i_{p+1} = j_2 = \cdots = j_d$.
Since $i_p \le i_{p+1} < \ell_1$, we have $i_1 = i_{p+1} = j_2 = \cdots = j_d$.
This contradicts the fact that 
at least one of $i_{p+1}$ and $ j_d$ is greater than $i_1$.
Thus, $f$ is nonzero.
Since the smallest variable appearing in $f$ is $y_{k_3} $,
the initial monomial of $f$ is $y_{k_1} y_{k_2}$.


\noindent
{\bf Case 2.} ($i_1 \notin \{  i_{p+1}, j_2\}$.)
Suppose that $i_{p+1} = j_2 = \cdots = j_d$.
Since 
$i_1 < i_{p+1} = j_2 = \cdots = j_d < \ell_1$, it follows that
$
f = 
y_{k_1} y_{k_2}
-  
y_{k_3} y_{k_4}
\in 
I_{\Pd}
$
is nonzero, where
$C_{k_3} = \{x_{i_1}, \ldots, x_{i_p}, x_{i_{p+1}}, \ldots, x_{i_{p+1}}\}$ and
$C_{k_4} = \{x_{i_{p+1}}, \ldots, x_{i_{p+1}}, x_{\ell_1}, \ldots, x_{\ell_q}\}$.
Since the smallest variable appearing in $f$ is $y_{k_3} $,
the initial monomial of $f$ is $y_{k_1} y_{k_2}$.
Therefore, there exists $j_s$ such that $i_{p+1} \neq j_s$.

Suppose that $j_s = \ell_1 = \cdots = \ell_q$.
Since $i_{p+1} < \ell_1$ and $i_1 < j_2$, it follows that
$
f = 
y_{k_1} y_{k_2}
-  
y_{k_3} y_{k_4}
\in 
I_{\Pd}
$
is nonzero, where
$C_{k_3} = \{x_{i_1}, \ldots, x_{i_p}, x_{i_{p+1}}, x_{\ell_1}, \ldots, x_{\ell_1}\}$ and
$C_{k_4} = \{x_{\ell_1}, x_{j_2}, \ldots, x_{j_d}\}$.
Since the smallest variable appearing in $f$ is $y_{k_3} $,
the initial monomial of $f$ is $y_{k_1} y_{k_2}$.
Therefore, there exists $\ell_t$ such that $\ell_t \neq j_s$.

Thus, $i_1$, $i_{p+1}$, $j_s$, and $\ell_t$ are distinct integers such that
$i_1 < j_s$, $i_{p+1} < \ell_t$,
and that
$\{x_{i_1} , x_{i_{p+1}} \}$,
$\{x_{i_1} , x_{\ell_t} \}$, and
$\{x_{i_{p+1}} , x_{j_s} \}$ are edges of $G_P$.
Since $x_1, \ldots, x_n$ is a strong perfect elimination ordering,
$\{ x_{j_s}, x_{\ell_t}  \}$ is an edge of $G_P$.
We now show that, for any $2 \le s' \le d$, 
$x_{j_{s'}}$ and $x_{\ell_t}$ are comparable.
We may assume that $ \ell_t \neq j_{s'} $.
If 
$j_{s'} = i_{p+1}$, then $x_{j_{s'}}$ and $x_{\ell_t}$ are comparable
since $P'$ is a multichain.
If $j_{s'} \neq i_{p+1}$, then $x_{j_{s'}}$ and $x_{\ell_t}$ are comparable by the same argument of $j_s$.
Thus, $x_{j_{s'}}$ and $x_{\ell_t}$ are comparable, and hence $\{x_{\ell_t}, x_{j_2}, \ldots, x_{j_d}\}$ is a multichain of $P$.
Since $i_{p+1} \neq \ell_t$ and $i_1 < j_2$, it follows that
$
f = 
y_{k_1} y_{k_2}
-  
y_{k_3} y_{k_4}
\in 
I_{\Pd}
$
is nonzero, where
$C_{k_3} = \{x_{i_1}, \ldots, x_{i_p}, x_{i_{p+1}}, x_{\ell_1}, \ldots, x_{\ell_{t-1}}, x_{\ell_{t+1}},\ldots, x_{\ell_q}\}$ and
$C_{k_4} = \{x_{\ell_t}, x_{j_2}, \ldots, x_{j_d}\}$.
Since the smallest variable appearing in $f$ is $y_{k_3} $,
the initial monomial of $f$ is $y_{k_1} y_{k_2}$.
\end{proof}

Let $A$ be a configuration.
The {\it initial ideal} of $I_A$ is an ideal generated by the initial monomial of 
the nonzero polynomials in $I_A$.
It is known that, if $I_A$ has a squarefree initial ideal, then $K[y_1,\ldots,y_m]/I_A$ is normal.
See \cite{dojoEN, Sturmfels}.
The toric ideals of
algebras of Veronese type \cite[Theorem 14.2]{Sturmfels}
and
algebras of Segre--Veronese type \cite{OH2, AHOT}
have a squarefree initial ideal.
On the other hand, $I_{\Pd}$ has no squarefree initial ideal
except for some trivial cases.
Let $P$ be a poset with ${\mathcal  M}_d(P) = \{C_1,\ldots, C_m\}$, and let
\begin{eqnarray*}
{\mathbb Z}A_P &=& \left\{ \sum_{i=1}^m z_i \rho(C_i) : z_i \in {\mathbb Z} \right\},\\
{\mathbb Z}_{\ge 0} A_P &=& \left\{ \sum_{i=1}^m z_i \rho(C_i) : 0 \le z_i \in {\mathbb Z} \right\},\\
{\mathbb Q}_{\ge 0} A_P &=&\left\{ \sum_{i=1}^m q_i \rho(C_i) : 0 \le q_i \in {\mathbb Q} \right\}.
\end{eqnarray*}
It is known \cite[Theorem 13.5]{Sturmfels} that 
$K[y_1,\ldots, y_m]/I_{{\mathcal  M}_d(P)}$ is normal
if and only if ${\mathbb Z}_{\ge 0} A_P = {\mathbb Z} A_P \cap {\mathbb Q}_{\ge 0} A_P$.
Proposition \ref{sqfnormal} shows that $I_{{\mathcal  M}_d(P)}$ 
is different from known classes of toric ideals 
associated with subconfigurations of Veronese subrings.

\begin{Proposition}
\label{sqfnormal}
Let $P$ be a poset.
Then, the following conditions are equivalent:

\begin{enumerate}
\item[(i)]
$P$ is a disjoint union of chains;

\item[(ii)]
$I_{{\mathcal  M}_d(P)}$ is the toric ideal of the
tensor product of the Veronese subrings;

\item[(iii)]
$I_{{\mathcal  M}_d(P)}$ has a squarefree initial ideal;

\item[(iv)]
$K[y_1,\ldots, y_m]/I_{{\mathcal  M}_d(P)}$ is normal.

\end{enumerate}
Moreover, the normalization of $K[y_1,\ldots, y_m]/I_{{\mathcal  M}_d(P)}$ 
is the tensor product of the Veronese subrings.
\end{Proposition}

\begin{proof}
The implication
(ii) $\Longrightarrow$ (iii) $\Longrightarrow$ (iv)
is known.
We will show 
(i) $\Longrightarrow$ (ii)
and 
(iv) $\Longrightarrow$ (i).
If $P$ is a disjoint union of posets $P_1$ and 
$P_2$,
then ${\mathcal  M}_d(P)$
is the disjoint union of $ {\mathcal  M}_d(P_1) = \{C_1,\ldots, C_\ell\}$
and $ {\mathcal  M}_d(P_2) = \{C_{\ell+1},\ldots, C_m\}$.
It then follows that
$K[y_1,\ldots, y_m]/I_{{\mathcal  M}_d(P)}$ is the tensor product
of $K[y_1,\ldots, y_\ell]/I_{{\mathcal  M}_d(P_1)}$
and $K[y_{\ell+1},\ldots, y_m]/I_{{\mathcal  M}_d(P_2)}$.
Thus, we may assume that $P$ is connected,
i.e., the comparability graph $G$ of $P$ is connected.

{\bf (i) $\Longrightarrow$ (ii)}
Suppose that $P$ is a chain.
Then, $I_{{\mathcal  M}_d(P)}$ is the toric ideal of the $d$-th Veronese subrings.


{\bf (iv) $\Longrightarrow$ (i)}
Since $G$ is connected, for each $1 \le j < k \le n$,
there exists a walk $x_j = x_{i_0}, x_{i_1}, \ldots , x_{i_r} = x_k$ in $G$.
Then, 
$
-{\bf e}_j + {\bf e}_k = 
d {\bf e}_{i_1}  + \cdots + d {\bf e}_{i_r}
-
\sum_{s=1}^r
({\bf e}_{i_{s-1}} + (d-1) {\bf e}_{i_s})
$ belongs to ${\mathbb Z} A_P$.
Hence, it follows that 
$${\mathbb Z}A_P = 
\left\{
(z_1, \ldots ,z_n)^T \in {\mathbb Z}^n
 : 
\sum_{i=1}^n z_i =d
\right\}.$$
Moreover, since $d {\bf e}_1, \ldots, d {\bf e}_n$
belong to $\{\rho(C_1),\ldots, \rho(C_m)\}$, we have
$$
{\mathbb Q}_{\ge 0} A_P = 
\left\{
(q_1, \ldots ,q_n)^T \in {\mathbb Q}^n
 : 
q_1, \ldots, q_n \ge 0
\right\}
.$$
Thus, 
$${\mathbb Z}A_P  \cap {\mathbb Q}_{\ge 0} A_P  = 
\left\{
(z_1, \ldots ,z_n)^T \in {\mathbb Z}^n
 : 
z_1,\ldots, z_n \ge 0,
\sum_{i=1}^n z_i =d
\right\}.$$
If $x_j$ and $x_k$ is not comparable,
then $(d-1) {\bf e}_j + {\bf e}_k
$ belongs to ${\mathbb Z}A_P  \cap {\mathbb Q}_{\ge 0} A_P
$, and does not belong to ${\mathbb Z}_{\ge 0}A_P$.
Hence 
$K[y_1,\ldots, y_m]/I_{{\mathcal  M}_d(P)}$ is not normal.

On the other hand, 
${\mathbb Z}A_P  \cap {\mathbb Q}_{\ge 0} A_P =  {\mathbb Z}_{\ge 0}A_{P'} $,
where $P'$ is a chain of length $n-1$.
Hence,  the normalization of $K[y_1,\ldots, y_m]/I_{{\mathcal  M}_d(P)}$ 
is the $d$-th Veronese subring.
\end{proof}

\section{The toric ideal $I_{\Ptwo}$ and edge rings}

In this section, we give some examples and remarks.
In particular,
we apply the results in Section 1 to a toric ring arising from a graph.

\begin{Example}
\label{d=2}
The following binomials form a Gr\"obner basis of  $I_{\Ptwo}$ with respect to $<_{\rm rev}$
appearing in the proof of Theorem \ref{main}:
\begin{eqnarray*}
y_{i\ell} y_{jk} - y_{ik} y_{j \ell} & & 
(i < j < k < \ell \mbox{ and } \{x_i, x_\ell\}, \{x_j, x_k\}, \{x_i,x_k\}, \{x_j, x_\ell \}
\in \Ptwo)\\
y_{i\ell} y_{jk} - y_{ij} y_{k \ell} & & 
(i < j < k < \ell \mbox{ and } \{x_i, x_\ell\}, \{x_j, x_k\}, \{x_i,x_j\}, \{x_k, x_\ell\}
\in \Ptwo)\\
 y_{ik} y_{j \ell}- y_{ij} y_{k \ell}  & & 
(i < j < k < \ell \mbox{ and } \{x_i, x_k\}, \{x_j, x_\ell\}, \{x_i,x_j\}, \{x_k, x_\ell\}
\in \Ptwo)\\
y_{i j} y_{i k} - y_{ii} y_{j k} & & (i < j < k \mbox{ and } \{x_i, x_j,x_k\} \in {\mathcal M}_3(P))\\
y_{jj} y_{i k} - y_{ij} y_{j k} & & (i < j < k \mbox{ and }\{x_i, x_j,x_k\} \in {\mathcal M}_3(P))\\
y_{ik} y_{j k} - y_{kk} y_{i j} & & (i < j < k \mbox{ and } \{x_i, x_j,x_k\} \in {\mathcal M}_3(P))\\
y_{ij}^2 - y_{ii} y_{jj} & & (i< j \mbox{ and } \{x_i, x_j\} \in \Ptwo),
\end{eqnarray*}
where each variable $y_{ij}$ corresponds to $\{x_i, x_j\} \in \Ptwo$.
The initial monomial of each binomial is the first monomial.
\end{Example}

Let $G$ be a simple graph and let 
$A_G$ be the vertex-edge incidence matrix of $G$.
The toric ideal $I_{A_G}$ of $A_G$ is referred to as the toric ideal of 
the {\it edge ring} of the simple graph $G$
and has been well  studied (e.g., \cite{OH1, OH2, Vil}).
In particular, a graph theoretical characterization for $I_{A_G}$ generated by quadratic binomials
is given in \cite[Theorem 1.2]{OH2}.
Let $\overline{A_G} =(\ 2 E_n \  | \  A_G \ )$, where $E_n$ is an identity matrix.
Then, we can regard ${\overline{A_G}}$ as 
a configuration arising from a non-simple graph whose edge set is
$E(G) \cup L$, where $L$ is the set of loops at vertices of $G$.
The edge rings of non-simple graphs are studied in, e.g., \cite{OHnormal}.
See also \cite[Remark 4.18]{RTT}.
If $G$ is the comparability graph of a poset $P$,
then $I_{\Ptwo} = I_{\overline{A_G}}$ holds.
The discussion in the proof of Theorem \ref{main} shows the following:

\begin{Theorem}
\label{stronglychordal}
Let $G$ be a graph. 
Then the following conditions are equivalent:
\begin{enumerate}
\item[(i)]
$G$ is strongly chordal;
\item[(ii)]
$I_{\overline{A_G}}$ has a quadratic Gr\"obner basis;
\item[(iii)]
$I_{\overline{A_G}}$ is generated by quadratic binomials.
\end{enumerate}
\end{Theorem}

\begin{proof}
The discussion in proof of Theorem \ref{main} shows that (i) $\Longrightarrow$ (ii).
In addition, (ii) $\Longrightarrow$ (iii) holds in general.
Hence, it is enough to show  (iii) $\Longrightarrow$ (i).
Suppose that $I_{\overline{A_G}}$ is generated by quadratic binomials.
From the proof of Theorem \ref{main}, it follows that $G$ has no induced even cycles.
Suppose that $G$ has an induced odd cycle $C=(v_1,v_2,\ldots, v_{2 \ell +1})$ of length $2 \ell +1 \geq 5$.
Since $C$ has no chords,
$
 ( {\bf e}_1 +{\bf e}_{2\ell+1})  + \sum_{k=1}^{\ell} ( {\bf e}_{2k -1 } +{\bf e}_{2k}) = 
2{\bf e}_1 + 
\sum_{k=1}^{\ell} ( {\bf e}_{2k} +{\bf e}_{2k+1})
$
satisfies the condition in Lemma \ref{lemma1}, which is a contradiction.
Thus, $G$ is chordal.
Suppose that $G$ is not strongly chordal.
By \cite[Theorem 7.2.1]{graphbook}, a sun graph $S_\ell$ ($\ell \ge 3$) is an induced subgraph of $G$.
Here $S_\ell$ is a graph whose edge set is
$E(S_\ell) = \{\{v_1,v_2\}, \{v_2,v_3\}, \ldots, \{v_{2\ell-1},v_{2 \ell} \},\{v_1,v_{2 \ell}\}\}
\cup E(G')$, where $G'$ is a graph on the vertex set $\{v_2,v_4,\ldots, v_{2\ell}\}$.
Since $\{v_1,v_3,\ldots, v_{2\ell-1}\}$ is independent in $G$, 
$
\sum_{k=1}^{\ell} ( {\bf e}_{2k -1 } +{\bf e}_{2k}) = 
({\bf e}_1 + {\bf e}_{2\ell}) +
\sum_{k=1}^{\ell-1} ( {\bf e}_{2k} +{\bf e}_{2k+1})
$
satisfies the condition in Lemma \ref{lemma1},  which is a contradiction.
Therefore, $G$ is strongly chordal, as desired.
\end{proof}

\begin{Example}
Let $G$ be the comparability graph of a poset $P = \{x_1, x_2, x_3, x_4, x_5\}$ 
whose maximal chains are $x_1 < x_2 < x_3$ and $x_1 < x_4 < x_5$.
Then $G$ is strongly chordal.
However, the toric ideal $I_{A_G}$ is a principal ideal generated by a binomial of degree $3$.
This example shows that, if we define $\Pd$ as the set of all {\it chains} $x_{i_1} < \cdots < x_{i_d}$ of $P$,
then the statement of Theorem \ref{main} does not hold.
\end{Example}

\end{document}